\documentclass[a4paper,12pt]{article}
\usepackage[utf8]{inputenc}
\usepackage[english]{babel}
\usepackage{amsmath}
\usepackage{amssymb}
\usepackage{latexsym}
\usepackage{graphics,graphicx}
\usepackage{color}
\usepackage{theorem}
\usepackage{psfrag}
\usepackage{hyperref}

\newcommand{\R}{\mathbb{R}}

\newtheorem{theorem}{Theorem} [section]
\newtheorem{proposition}[theorem]{Proposition}
\newtheorem{remark}[theorem]{Remark}
\newtheorem{lemma}[theorem]{Lemma}
\newtheorem{definition}[theorem]{Definition}
\newtheorem{corollary}[theorem]{Corollary}

\catcode`@=11 \@addtoreset{equation}{section} \catcode`@=12

\newenvironment{proof}[1][Proof]{\noindent \textbf{#1.} }
{\  \rule{0.5em}{0.5em}\par \medskip}

\title{PDEs in moving time dependent domains
\thanks{Partially supported by Project MTM2009--07540, MEC and   GR58/08
  Grupo 920894, UCM, Spain.}
}

\author{Fernando Cortéz \\ An\'{\i}bal Rodr\'{\i}guez-Bernal${}^{\dag}$}

\date{}
\begin{document}

\maketitle

\begin{center}
Departamento de Matem\'atica Aplicada \\
Universidad Complutense de Madrid, \\ Madrid 28040,  SPAIN \\ and \\
${}^{\dag}$ Instituto de Ciencias Matem\'aticas \\
CSIC-UAM-UC3M-UCM
\end{center}

\begin{center}
Dedicated to Professor M.G. Velarde\\
 in occasion of his 70th birthdate.
  
\end{center}

\begin{abstract}

In this work we study partial differential equations defined in a
domain that moves in time according to the flow of a given ordinary
differential equation, starting out of a given initial domain. We
first derive a formulation for a particular case of partial
differential equations known as balance equations. For this kind of
equations we find the equivalent partial differential equations in the
initial domain and later we study some particular cases with and
without diffusion. We also analyze general second order differential
equations, not necessarily of balance type. The equations without
diffusion are solved using the characteristics method. We also prove
that the diffusion equations, endowed with Dirichlet boundary
conditions and initial data, are well posed in the moving domain. For
this we show that the principal part of the equivalent equation in the
initial domain is uniformly elliptic. 
We then prove a version of the weak  maximum principle for an equation in a moving domain.
Finally we perform suitable energy estimates in the moving domain and
give sufficient conditions for the solution to converge to zero as
time goes to infinity 

\end{abstract}

\section{Introduction}

In a standard setting for many partial differential equations of
mathematical physics, one
usually assumes that the physical process being described occurs in a
fixed domain of the physical space. This includes many equations
describing the motion of fluids for example, despite the fact that
particle fluids and hence fluid subdomains actually move with time. Of
course there are some other problems,such as free boundary problems,
in which the physical domain of the PDE changes with time. In all
these problems the motion of particles or subdomains occurs according
to an unknown velocity field with is actually one of the main unknowns
of the problem.

In this paper we assume some intermediary situation in which each
point of a given initial domain $\Omega_{0}\subset \R^{n}$, moves in
time according to some prescribed autonomous vector field. Hence at
later times the domain $\Omega_{0}$ evolves into a diffeomorphic
domain $\Omega (t)$ (which is not excluded to coincide with
$\Omega_{0}$ itself!). In particular, topological properties of the
domain are preserved along time. However the geometrical evolution of
the domain can be very complex; for example one can consider the
evolution of the open set  $\Omega_{0}$ in $\R^{3}$ with the vector field of the
Lorenz equations in a chaotic regime. 

Our goal is the to describe some sensible class of PDEs to be consider
in such a family of moving domains. We choose then to describe balance
equations in moving domains, which result from conservation principles
and which have natural applications to conservation of mass, momentum,
energy etc. For such equations one must then give some suitable
definition of  solution.  

After giving a convenient meaning of solution for both balance and
general parabolic equations, we prove that such equations can be
solved using available results. 

Then we investigate, on some particular, although significative
examples of equations in moving domains,  basic tools in the analysis
of parabolic equations such as the (weak) maximum principle and
energy estimates. In particular we obtain sufficient conditions on the
equations and on the moving domains, that guarantee that the
solutions converge to zero as time goes to  infinity. 
  
\section{Moving domains}

We assume that each point $x$ of an  original given domain (smooth
open set)  $\Omega_0
\subset \R^{n}$,
starting at time $t=0$ moves following a curve $t \  \longmapsto
Y(t;x)$, in $\R^{n}$. Moreover we assume this curve is a solution of
the autonomous system of ODEs 
\begin{equation}
 \left\{
\begin{array}{l}
\dot{Y}(t;x)= \vec{V}(Y(t;x)) \\
Y(0;x)=x 
\end{array}\right.
\label{EBDSVT1}
\end{equation}
for some given smooth vector velocity field $\vec{V} : \R^{n}
\longrightarrow  \R^{n}$. 
Even more and for simplicity we assume that all solutions of
(\ref{EBDSVT1}) are defined for all $t\in \R$.

Hence, for $t\in \R$,  we have a deformation map 
\begin{displaymath}
\phi(t) :  \R^{n} \longrightarrow \R^{n}, \quad \phi(t)z = Y(t;z)
\end{displaymath} 
which is a diffeormorphism that  satisfies the group properties
$\phi(0)= I$, $\phi(t+s)=\phi(t)\circ \phi(s)$ for all $t,s \in
\R$. In particular $\phi(-t)$ is the inverse of $\phi(t)$.

Therefore, the original domain $\Omega_{0}$ is deformed into the
domains 
\begin{displaymath} \label{22}
\Omega(t) = \phi(t) \Omega_0  \ \ \ \ \ \ t \in \R
\end{displaymath}
and the boundaries satisfy $\partial\Omega(t) =\phi(t) \ \partial \Omega_0$.
Also, any smooth subdomain  $W_0$  of  $\Omega_0$ is also deformed
into 
\begin{displaymath}
W(t) = \phi(t) W_0, \quad  t \in \R
\end{displaymath}
and its boundary is given by $\partial W(t) = \phi(t) \partial W_0$.

The next results gives geometrical information about the deformations
above. 
\begin{lemma} \label{erestu}

With the above notations, for  $x_0 \in \partial \Omega_0$ then
$\phi(t)x_0 \in \Omega(t)$ and 
\begin{displaymath}
D\phi(t)(x_0)
\end{displaymath}   
is an isomorphism in  $\R^{n}$ that transforms the tangent plane in
$x_0\in \partial\Omega_0$, that we denote  $T_{x_0} \partial\Omega_0$,
into  the tangent plane to  $\partial\Omega(t)$ at  $\phi(t)x_0$,
$T_{\phi(t)x_0} \partial\Omega(t)$. 

\end{lemma}
\begin{proof}
Just note that if  $z(s)$ is a curve in  $\partial\Omega_0$ with
$z(0)= x_0$, then  $z'(0)=v_0$ is a tangent vector at $x_{0}$ (and
conversely). Hence,  $w(s)= \phi(t)(z(s))$ is a curve in  $
\partial\Omega(t)$, with  $w(0)=y_0$ and 
\begin{displaymath}
w'(0)=D\phi(t)(x_0) v_0
\end{displaymath} 
is a tangent vector at  $\partial\Omega(t)$ .
\end{proof}

We also recall the following 

\begin{definition}
  A matrix  $\eta(t)$ is a fundamental matrix of the linear system 
\begin{equation}
X'(t)=A(t) X(t) 
\label{gioo}
\end{equation}
iff each column of  $\eta(t)$ is a solution of  (\ref{gioo}) and
$\eta(t)$ is nonsingular. 
\end{definition}
Observe that in particular, $\eta'(t)=A(t)\eta(t)$.   Then we have 

\begin{lemma} 

If  $\eta(t)$ is a fundamental matrix of  (\ref{gioo}), then 
\begin{displaymath}
\gamma(t)=\left( \displaystyle\eta^{-1}(t) \right)^{*}=
\left(\displaystyle\eta^{*}(t)\right)^{-1} 
\end{displaymath}
is a fundamental matrix of the adjoint system 
\begin{displaymath}
Y'(t)=-A^{*}(t) Y(t)
\end{displaymath}
where  * denotes the adjoint matrix. 
\end{lemma}
\begin{proof}
Differentiate in 
\begin{displaymath}
\eta^{-1}(t) \circ \eta(t)= I 
\end{displaymath}
and use  (\ref{gioo}).  
\end{proof}

The following result is obtained from classical results in ODEs, see
\cite{H}.

\begin{proposition}
\mbox{}

\noindent i) For $x \in \R^{n}$,  $D\phi(t)x$ is a fundamental matrix of 
\begin{displaymath} \label{geova1}
  \dot{Z}(t)=A(t) Z(t)
\end{displaymath}
and $D\phi(0)=I$, where $A(t)=D\vec{V}(\phi(t)x)$. 

\noindent ii) Denote 
\begin{displaymath}
 \left | K(x,t) \right |= det(D\phi(t)x), \quad x \in \R^{n} ¡
\end{displaymath}
then we have the  \textbf{Abel--Liouville--Jacobi} formula 
\begin{displaymath} \label{EBDSVT2}
\frac{\partial}{\partial t} \left | K(x,t) \right | = tr\big( D\vec{V}\big) (\phi(t)x)  \left | K(x,t) \right | 
= div \big(  \vec{V}\big) (\phi(t)x) \left | K(x,t) \right |  
\end{displaymath}
hence 
\begin{displaymath} \label{geova}
\left | K(x,t) \right |= e^{\int_{0}^{t} div\vec{V}(\phi(s)x)\, ds}. 
\end{displaymath}

In particular, for  $t\in[-T,T]$ there exist $C_1(T),C_2(T)$ 
such that 
\begin{equation}
0 < C_1(T)\leq \left | K(x,t) \right | \leq C_2(T) \ \ \ \forall x \in \Omega_{0} \ \ \forall  t \in [0,T].
\label{ACDC}
\end{equation}
\end{proposition}

\begin{remark}

Observe that if  $W_0 \subset \Omega_0$ and  $W(t)= \phi(t)W_0$ then
the measure  of  $W(t)$ satisfies 
\begin{displaymath}
\left | W(t) \right | = \displaystyle\int_{W(t)}^{} 1  \, dy =
\displaystyle\int_{W_0}^{} \left | K(x,t) \right |  \, dx
=\displaystyle\int_{W_0}^{}   e^{\int_{0}^{t} div
  \vec{V}(\phi(s)(x))\, ds} \, dx   . 
\end{displaymath}

In particular, if  $div(\vec{V})=0$ then the measure is preserved,
that is, 
\begin{displaymath}
\left | W(t) \right | = \left |  W_0 \right |   \ \ \ \ \ \ \ \forall
\  W_0 \subset \Omega_0 \ \ \ \ \forall t\in \R . 
\end{displaymath}
Also, if $ div(\vec{V}) \leq - d_0 < 0$ at every point, then 
\begin{displaymath}
 \left | W(t) \right | \leq \left |  W_0 \right | \ \  e^{-d_0 t}
\end{displaymath}
and we say the flow of (\ref{EBDSVT1}) is contractive. 

Finally if   $div(\vec{V})  \geq  d_0 > 0$ at every point, then 
\begin{displaymath}
 \left | W(t) \right | \geq \left |  W_0 \right | \ \  e^{d_0 t}
\end{displaymath}
and we say the flow if expansive.

For example for a linear flow, that is,  $\vec{V}(x) = M x$ for a
given matrix  $M$, we have  
\begin{displaymath}
div(\vec{V})= tr(M)= \displaystyle\sum_{i=1}^n \mu_i = d_0
\end{displaymath}   
is the trace of $M$, that is the sum of all eigenvalues of $M$.

\end{remark}

Then we have the following result that complements Lemma
\ref{erestu}. 

\begin{corollary} \label{Angie}
Assume  $ x_0 \in \partial\Omega_0$ and  consider  $y_0 =
\phi(t)x_{0} \in\partial\Omega(t)$. Then if  $\vec{n}(x_0)$ is an unitary  outward
normal vector  to  $\Omega_0$ at  $x_0$ then 
\begin{displaymath}
 N(y_0) =((D\phi(t)x_0)^{*})^{-1} \vec{n}(x_0)
\end{displaymath}
is an outward vector at  $y_0$. That is,  $((D\phi(t))^{*} x_0)^{-1}$
is a linear isomorphism in  $\R^{n}$ that transforms the normal space
at  $x_0 \in \partial\Omega_0$, which we denote,  $N_{x_0}$, into the
normal space to  $\Omega(t)$ at  $y_0\in \partial\Omega(t)$, which we
denote  $N_{y_0}$.

\end{corollary}
\begin{proof}
From Lemma  \ref{erestu} a normal vector at  $y_0=\phi(t)(x_0)\in
\partial\Omega(t)$ , $\vec{n}$, must satisfty 
\begin{displaymath}
<\vec{n} , D\phi(t)x_0 \vec{\tau}>= 0 \ \ \ \forall \vec{\tau}\in T_{x_0} \partial\Omega_0   
\end{displaymath}  
which reads  
\begin{displaymath}
<(D\phi(t)x_0)^{*} \vec{n} , \vec{\tau} >= 0 \ \ \ \forall
\vec{\tau}\in T_{x_0} \partial\Omega_0.   
\end{displaymath}
Hence we can take  $\vec{n}$ such that $((D\phi(t)x_0)^{*}) \vec{n} =
\vec{n}(x_0)$ 
which gives the result. 
\end{proof}

\section{Balance equations}

The following notations will be used throughout the paper. 

\begin{definition}
If  for some $T>0$, $f$ is defined in 
\begin{displaymath}
f :\cup_{t \in (-T,T)} \Omega(t) \times \{t\} \longrightarrow \R,
\quad (y,t) \longmapsto f(y,t)
\end{displaymath}
then we define  $\overline{f}$ in $\Omega_{0}$ as 
\begin{displaymath}
\overline{f} :\Omega_0 \times (-T,T) \longrightarrow \R, \quad 
\overline{f}(x,t)=f(\phi(t)x,t)
\end{displaymath}

\end{definition}

Consider  $W(t)=\phi(t)W_0\subset \Omega(t)$, a sufficiently smooth
region with boundary $\partial W(t)$. Then the time variation of the
amount of $T$ in $W(t)$ is given by 
\begin{displaymath}
\frac{d}{d t}\displaystyle\int_{W(t)}^{} T(y,t)\, dy
\end{displaymath} 
which is computed below. Note that this is the classical Reynolds
Transport theorem, \cite{LOA,DG,CAJ}.

\begin{proposition}
With the notations above, we have that 
\begin{displaymath}
\frac{d}{d t} \displaystyle\int_{W(t)}^{} T(y,t)\, dy 
\end{displaymath}
can be written by either one of the following equivalent expressions 
\begin{equation}
\displaystyle\int_{W_0}^{} \frac{\partial \ \overline{T}}{\partial
  t}(x,t) \left | K(x,t) \right | \, dx + \displaystyle\int_{W_0}^{}
\overline{T}(x,t) \ \overline{ div \vec{V}}(x,t) \left | K(x,t) \right
|\, dx  
\label{EBDSV3}    
\end{equation}
 or 
\begin{equation}
\displaystyle\int_{W(t)}^{}\frac{\partial T}{\partial t}(y,t)\, dy +
\displaystyle\int_{W(t)}^{} div_y(T(y,t) \ .\ \vec{V}(y))\, dy   
\label{EBDSV4}
\end{equation}
or 
\begin{equation}
\displaystyle\int_{W(t)}^{} \frac{\partial T}{\partial t}(y,t) \, dy
+\displaystyle\int_{\partial W(t)}^{} T(y,t) \vec{V}(y)   \,
d\vec{s}. 
\label{EBDSV5}
\end{equation}
\end{proposition}

Now we will derive the Balance Equations for the quantity $T(y,t)$. In
fact we have 
\begin{displaymath}  \label{EBDSVT8}
 \frac{d}{d t}\displaystyle\int_{W(t)}^{} T(y,t)\, dy =
 \displaystyle\int_{W(t)}^{}f(y,t) \,dy - \displaystyle\int_{\partial
   W(t)}^{} \vec{J} \,d\vec{s}  
\end{displaymath} 
where  $f(y,t)$ represents the rate of production/consumption of $T$
per unit volume in $W(t)$ and    $\vec{J}$  is the vector field of the
flow of $T$ across the boundary of  $W(t)$. Hence the divercence
theorem leads to 
\begin{equation}
 \frac{d}{d t}\displaystyle\int_{W(t)}^{} T(y,t)\, dy =
 \displaystyle\int_{W(t)}^{}f(y,t) \,dy - \displaystyle\int_{W(t)}^{}
 div_y\vec{J} \,dy  
\label{EBDSVT9}
\end{equation} 

Hence, (\ref{EBDSVT9}) and the Proposition above leads to

\begin{proposition}
Under the assumptions and notations above, the magnitud $T$ satisfies
the balance equations in the moving domains, if and only if the
following equivalent conditions are satisfied: 
\begin{equation}
\frac{\partial T}{\partial t}(y,t) +div_y(T(y,t) \ .\ \vec{V}(y)) =
f(y,t)  -div_y(\vec{J}), \quad  y\in \Omega(t) , \ t>0
\label{EBDSVT11}
\end{equation}
or 
\begin{equation}
\frac{\partial}{\partial t} \overline{T}(x,t) +\overline{T}(x,t) \ \
\overline{div (\vec{V})}(x,t) =\overline{f}(x,t)
-\overline{div_y(\vec{J})}(x,t) \quad x\in \Omega_0, \ t>0 . 
\label{EBDSVT12}
\end{equation} 

\end{proposition}
\begin{proof} 
First, equating  (\ref{EBDSV4})  and  (\ref{EBDSVT9}) we get  
\begin{displaymath}
\displaystyle\int_{W(t)}^{} \left(\displaystyle\frac{\partial
    T}{\partial t}(y,t)+ div_y(T(y,t) \vec{V}(y)) \right) \, dy =
\displaystyle\int_{W(t)}^{}(f(y,t) - div_y\vec{J}) \,dy .  
\end{displaymath} 
Since  $W(t)=\phi(t)(W_0)$,   $\phi(t)$ is a diffeormorphism and
$W_{0}$ is arbitrary, we get (\ref{EBDSVT11}). 

Now,   using $y = \phi(t)x$ we get in the right
hand side  of (\ref{EBDSVT9}) 
\begin{displaymath} \label{EBDSVT10}
 \displaystyle\int_{W_0}^{}
\overline{f}(x,t) \left | K(x,t) \right |  \, dx -
\displaystyle\int_{W_0}^{} \overline{div_y(\vec{J})}(x,t) \left |
  K(x,t) \right | \, dx  , 
\end{displaymath}
equating to  (\ref{EBDSV3})  and using that $W_{0}$ is arbitrary, we
get (\ref{EBDSVT12}). 
\end{proof}

\section{Boundary and initial conditions}

As we consider Dirichlet boundary conditions and using 
\begin{displaymath}
y=\phi(t)x, \qquad \Omega(t) = \phi(t) \Omega_0,   \qquad \partial
\Omega(t)  = \phi(t)(\partial \Omega_0)
\end{displaymath}
then 
\begin{displaymath}
T(y,t)=0 \ \ \ \forall \ y\in \partial \Omega(t) \ \ \Leftrightarrow \ \ \overline{T}(x,t)=0 \ \ \ \forall \ x\in \partial \Omega_0
\label{Marci}
\end{displaymath}

As for the initial condition we have, since $\phi(0)=I$,   
\begin{displaymath}
T(y,0)= T_0(y) \ \ \forall \ y\in \Omega_0   \ \
\Leftrightarrow \ \ \overline{T}(x,0)= T_0(x) \ \ \forall  \ x\in
\Omega_0 . 
\label{Marci1}
\end{displaymath}

Thus,  (\ref{EBDSVT11}) and  (\ref{EBDSVT12}), with boundary and
initial conditions read, respectively, 
\begin{equation}
\left\{
\begin{array}{cl}
\frac{\partial T}{\partial t}(y,t) +div_y(T(y,t) \ .\ \vec{V}(y)) =
f(y,t)  -div_y(\vec{J}) \ \ \ \ y\in \Omega(t) &\mbox{ } \\ 
&\mbox{ } \\
\  \ \ \ T(y,t)= 0    \ \ y \in \partial \Omega(t)\ \ \  \forall t \ \
\ \ \ \ \ \ \ \  T(y,0)= T_0(y)  \ \ \ \ \  y \in \Omega_0     &\mbox{
}   
\end{array}\right.
\label{EBCDC1}
\end{equation}
\begin{equation}
\left\{
\begin{array}{cl}
\frac{\partial}{\partial t} \overline{T}(x,t) +\overline{T}(x,t) \ \
\overline{div (\vec{V})}(x,t) =\overline{f}(x,t)
-\overline{div_y(\vec{J})}(x,t) \ \ \ x\in \Omega_0 &\mbox{ } \\ 
&\mbox{ } \\
\  \ \ \ \overline{T}(x,t)= 0    \ \ x \in \partial \Omega_0\ \ \
\forall t \ \  \ \ \ \ \ \ \ \  \overline{T}(x,0)= T_0(x)  \ \ \ \ \
x \in \Omega_0.     &\mbox{ }  
\end{array}\right.
\label{EBCDC2}
\end{equation}

In fact we use  (\ref{EBCDC2}) to define a solution of
(\ref{EBCDC1}), i.e.  
\begin{displaymath}
T(y,t)\ \  \mbox{satisfies}  \ \  (\ref{EBCDC1}) \ \ \Leftrightarrow \overline{T}(x,t) \ \ \mbox{satisfies} \ \ (\ref{EBCDC2}). 
\end{displaymath}

\section{Balance equations without diffusion}

\subsection{ No flux and no diffusion: pure inertia}

With the previous notations, assume $div_y(\vec{J})=0$ then the
following problems are equivalent 
\begin{equation}
\left\{
\begin{array}{cl}
\frac{\partial }{\partial t} T(y,t) +div_y(T(y,t) \ .\ \vec{V}(y)) = f(y,t) \ \ \ \ y\in \Omega(t)    &\mbox{ } \\
&\mbox{ } \\
T(y,t)=0  \ y\in\partial \Omega(t) \  \  \forall t \ \  \ \ \ \ \ \ \ \ \ \ \  T(y,0)= T_0(y)  \ \ \ \ \ \ \  y \in \Omega_0     &\mbox{ }  
\end{array}\right.
\label{EBDSVTSD1}
\end{equation}
and 
\begin{equation}
 \left\{
\begin{array}{cl}
\frac{\partial}{\partial t} \overline{T}(x,t) +\overline{T}(x,t) \ \ \overline{div (\vec{V})}(x,t) =\overline{f}(x,t) \ \ \ x\in \Omega_0    &\mbox{ } \\
&\mbox{ } \\
 \overline{T}(x,t)= 0  \ \ \ x\in\partial\Omega_0 \ \ \ \ \forall t \
 \  \ \ \ \ \ \ \ \ \ \ \  \overline{T}(x,0)= T_0(x)  \ \ \ \   x\in
 \Omega_0.    &\mbox{ }  
\end{array}\right.
\label{EBDSVTSD2}
\end{equation}
Then we have 

\begin{proposition}
With the notations above,  (\ref{EBDSVTSD1}) and  (\ref{EBDSVTSD2})
have a unique explicit solution given by 
\begin{displaymath}
T(y,t)=T_0(x) \  e^{-\int_{0}^{t} div_y \vec{V}(\phi(r)x) \, dr}  \
+ \ \displaystyle\int_{0}^{t}  \ e^{-\int_{s}^{t} div_y
  \vec{V}(\phi(r)x)  \, dr} \ f(y,s)  \, ds , \quad  y=\phi(t)x \in \Omega(t) 
\end{displaymath} 
and 
\begin{displaymath}
\overline{T}(x,t)=T_0(x) \  e^{-\int_{0}^{t} div \vec{V}(\phi(r)x) \,
  dr}  \ +   \displaystyle\int_{0}^{t}  \ e^{-\int_{s}^{t}
  div\vec{V}(\phi(r)x) \, dr} \ \overline{f}(x,s)  \, ds  , \quad x
\in \Omega_{0}, 
\end{displaymath}
respectively. 

\end{proposition}
\begin{proof}
  The solution of (\ref{EBDSVTSD2}) is obtained by solving a linear
  nonhomogeneous ODE
  \begin{displaymath}
Z'(t)+ P(t) Z(t)=h(t) , \quad   Z(0)= Z_0    
  \end{displaymath}
for each $x\in \Omega_{0}$. From this the solution of
(\ref{EBDSVTSD1}) is immediate. 
\end{proof}

\begin{remark}
Assume in particular that there are no source terms, that is,
$f=0$. Hence in  (\ref{EBDSVTSD1}) we have 
\begin{displaymath}
T(y,t)=T_0(x) \  e^{-\int_{0}^{t} div_y \vec{V}(\phi(r)x) \, dr}   ,
\quad  y=\phi(t)x 
\end{displaymath}

Thus, if moreover $div(\vec{V})=0$ then 
\begin{displaymath}
T(y,t)=T_0(x) \quad  y=\phi(t)x , 
\end{displaymath}
and   $T$ remains constant along the paths of the flow. 

On the other hand if the flow is expansive then  $T(y,t)$ decreases
along the paths of the flow, while it increases if the flow is
contractive. 

\end{remark}

\subsection{Flux and no diffusion: transport equations}

Below we use $\psi(t)= \phi^{-1}(t) = \phi(-t)$. 

\begin{proposition} 
If we assume  
\begin{displaymath}
\vec{J}(y,t)= \vec{a}(y,t) T(y,t) \ \ \ \ y\in \ \Omega(t) 
\end{displaymath}
with a $C^{1}$ scalar field 
\begin{displaymath}
\vec{a}: \R^{n}\times \R \ \ \rightarrow \ \ \R
\end{displaymath}
then the balance equations (\ref{EBCDC1}) and (\ref{EBCDC2}) read 
\begin{equation}
\left\{
\begin{array}{cl}
\frac{\partial}{\partial t} T(y,t) +div_y(T(y,t) \ .\ \vec{V}(y)) +
\nabla_y T(y,t) . \vec{a}(y,t)+ div_y(\vec{a}(y,t)) T(y,t)= f(y,t)\  \
\ y \ \in \Omega(t)  &\mbox{ } \\ 
&\mbox{ } \\
\ T(y,t)= 0  \ \ \ y\in \partial\Omega(t) \ \ \forall t \ \  \ \ \ \ \
\ \ \ \ \ \  T(y,0)= T_0(y)  \ \ \ \ \ \ \  y \in \Omega_0     &\mbox{
}   
\end{array}\right.
\label{EBDSVTCD1}
\end{equation}
 and 
\begin{equation}
\left\{
\begin{array}{cl}
\frac{\partial \overline{T}}{\partial t} (x,t)+ \overline{T}(x,t)
C(x,t)  + \nabla_x \overline{T}(x,t) \vec{b}(x,t)  =\overline{f}(x,t)
\ \ \ x\in \Omega_0 &\mbox{ } \\ 
&\mbox{ } \\
 \overline{T}(x,t)= 0 \   x\in \partial\Omega_0 \ \  \forall t  \ \ \
 \ \ \ \ \  \overline{T}(x,0)= T_0(x)  \ \ \ \   x\in \Omega_0
 &\mbox{ }   
\end{array}\right.
\label{EBDSVTCD2}
\end{equation}
which are equivalent, where  
\begin{displaymath}
C(x,t)= \overline{div_y(\vec{V})}(x,t)+
\overline{div_y(\vec{a})}(x,t), \quad 
\vec{b}(x,t)=   \overline{ D\psi(t) y  \cdot \vec{a}(y,t)}  .  
\end{displaymath}
\end{proposition}
\begin{proof}
Note that (\ref{EBDSVTCD1}) follows by direct computation from
(\ref{EBCDC1}) using 
\begin{displaymath}
div_y(\vec{a}(y,t) \ T(y,t))= \nabla_y T(y,t) . \vec{a}(y)+ T(y,t)
div_y(\vec{a}(y,t)) . 
\end{displaymath}

On the other hand, for (\ref{EBDSVTCD2}) we have to write
$div_y(a(y,t)T(y,t))$ in terms of $x$. For this we observe that since
$x=\psi(t)y$ we have $T(y,t)= \overline{T}(\psi(t)y,t)$ 
and then 
\begin{equation} \label{amankaya} 
\frac{\partial T}{\partial y_i}(y,t)=\displaystyle\sum_{j=1}^n
\frac{\partial \overline{T}}{\partial x_j} (x,t) \frac{\partial
  \psi_j(t) y}{\partial y_i}  
\end{equation}
and $\nabla_y T(y,t)= \nabla_x \overline{T}(x,t) D\psi(t)y$. 

Thus, $\nabla_y T(y,t) . \vec{a}(y,t)=\nabla_x \overline{T}(x,t)
\left(\displaystyle D\psi(t)y \  . \  \vec{a}(y,t) \right)$  and hence 
\begin{displaymath}
\overline{\nabla_y T(y,t). \vec{a}(y,t)} = \nabla_x \overline{T}(x,t)
\overline{ \left(\displaystyle D\psi(t)y   \cdot
    \vec{a}(y,t)\right)} = \nabla_x \overline{T}(x,t) \vec{b}(x,t) . 
\end{displaymath}
\end{proof}

Now we show that under some natural geometrical conditions
(\ref{EBDSVTCD2}) (and hence (\ref{EBDSVTCD1})) can be solved by using
characteristics. Note that we now disregard boundary conditions.

\begin{proposition}
Assume that for all time and  $y \in \partial\Omega(t)$, we have  
\begin{displaymath}
<\vec{a}(y,t) , \vec{n_0}(y) > \leq 0 
\end{displaymath}
where  $<\cdot,\cdot>$ is the scalar product and and  $\vec{n_0}(y)$
is the unit outward normal vector at  $y$. 

Then  (\ref{EBDSVTCD1}) and  (\ref{EBDSVTCD2}) have a unique solution.

\end{proposition}
\begin{proof}
For  (\ref{EBDSVTCD2}) we use the method of characteristics. Hence,
for  $x_0 \in \Omega_0$ we define curves defined on some interval $I$
containing $0$
\begin{displaymath} 
s \longmapsto
X(s) \ \in \Omega_0, \quad X(0)= x_{0}, 
\qquad 
 s \longmapsto t(s) \ \in \R^{+}, \quad t(0)=0, 
\end{displaymath}
and  $s \longmapsto Z(s)= \overline{T}(X(s),t(s))$. Then 
\begin{displaymath} \label{caracteristicas2} 
\frac{d}{d s} Z(s)=  \nabla_x \overline{T}(X(s),t(s))
X^{'}(s) + \frac{\partial}{\partial t}
\overline{T}(X(s),t(s)) t^{'}(s).   
\end{displaymath}

So from   (\ref{EBDSVTCD2}) we choose 
\begin{displaymath} \label{EBDSVTCD7}
t'(s) = 1, \quad t(0)=0,
\end{displaymath} 
\begin{displaymath} \label{EBDSVTCD8}
X'(s) = \vec{b}(X(s),t(s)), \quad  X(0)= x_0
\end{displaymath}
which gives $t(s)=s$ and 
\begin{equation}
X'(t) = \vec{b}(X(t),t), \quad  X(0)= x_0 \in \Omega_{0}, 
\label{mama}
\end{equation}
which has a solution because  $\vec{b}\in C^{1}( \R^{n})$. 

Hence, from   (\ref{mama}) and (\ref{EBDSVTCD2})  
\begin{displaymath} \label{caracteristicas4}
\left\{
\begin{array}{cl}
\frac{d}{d t} Z(t) + C(X(t),t) Z(t)   =\overline{f}(X(t),t)  &\mbox{ } \\
 \ \ \ \ \ \ \ \ \ \ \  Z(0)=  T_0(x_0)       &\mbox{ }  
\end{array}\right.
\end{displaymath}
whose solution is given by 
\begin{equation}
Z(t)=T_0(x_0) e^{-\int_{0}^{t} C(X(r),r) \, dr} +
\displaystyle\int_{0}^{t} e^{-\int_{s}^{t} C(X(r),r) \, dr}
\overline{f}(X(s),s) \, ds  . 
\label{pipis}
\end{equation}

In the computation above we need the solution of (\ref{mama}) not to
leave $\Omega_0$. Thus, if  $X(t)$ reaches the boundary of  $\Omega_0$
at time  $t_0$ at the point $y_0=x(t_0) \in \partial\Omega_0$, the
tangent vector to the characteristic curve at this point is
$X^{'}(t_0)= \vec{b}(x_0,t_0)$, 
and therefore if it points inward, that is, if  
\begin{equation} \label{eres}
<\vec{b}(x_0,t_0), \vec{n} (x_0)> \ \leq 0
\end{equation} 
then it will remain in $\Omega$. Note now that from  (\ref{eres}) 
\begin{displaymath}
<\vec{b}(x_0,t_0), \vec{n} (x_0)> =  <D\psi(t_0)y_0. \vec{a}(y_0,t_0),
\vec{n}(x_0) > = < \vec{a}(y_0,t_0) , (D\psi(t_0)y_0)^{*}
\vec{n}(x_0)> =   
\end{displaymath}
\begin{displaymath}
=<\vec{a}(y_0,t_0) ,((D\phi(t_0)x_0)^{*})^{-1} \vec{n}(x_0) >= <\vec{a}(y_0,t_0) , N(y_0)> \leq 0 
\end{displaymath}
where we have used Corollary  \ref{Angie} and the assumption of this
Proposition. 

With this  (\ref{pipis}) gives the values of the solution in the
moving domain. 
\end{proof}

\section{ Balance equations with diffusion} 

Recalling the equivalent equations  (\ref{EBCDC1}) and
(\ref{EBCDC2}) we have 

\begin{proposition} 

Assume the flux vector field is given by  
\begin{displaymath}
\vec{J}(y,t)=-k\nabla_y T(y,t)  \ \ \ \ y\in \ \Omega(t) 
\end{displaymath}
for some   $k > 0$, then (\ref{EBCDC1}) and
(\ref{EBCDC2}) read 
\begin{equation}
\left\{
\begin{array}{cl}
\frac{\partial }{\partial t} T(y,t)  + \nabla_y T(y,t).\vec{V}(y)+ T(y,t) div(\vec{V})(y) - k \Delta T(y,t)= f(y,t) \ \ \ t \in\Omega(t) &\mbox{ } \\
&\mbox{ } \\
\ T(y,t)= 0  \ \ \ y\in \partial\Omega(t) \ \ \forall t \ \  \ \ \ \ \ \ \ \ \ \ \  T(y,0)= T_0(y)  \ \ \ \ \ \ \  y \in \Omega_0     &\mbox{ }  
\end{array}\right.
\label{EBDSVTCDD1}
\end{equation}
and 
\begin{equation}
\left\{
\begin{array}{cl}
\frac{\partial  \overline{T}(x,t)}{\partial t} +\overline{T}(x,t) \
\overline{div (\vec{V})}(x,t)- k \left(
  \displaystyle\displaystyle\sum_{k,i=1}^n   a_{k,i}(x,t) \frac{\partial^{2} \
    \overline{T}(x,t)}{\partial x_k \partial x_i}  \ + \
  \displaystyle\sum_{i=1}^n \frac{\partial \overline{T}(x,t)}{\partial
    x_i}  . s_i(x,t) \right) = \overline{f}(x,t) &\mbox{ } \\ 
&\mbox{ } \\
 \overline{T}(x,t)= 0 \   x\in \partial\Omega_0 \ \  \forall t  \ \ \
 \ \ \ \ \  \overline{T(x,0)}= T_0(x)  \ \ \ \   x\in \Omega_0 &\mbox{
 }   
\end{array}\right.
\label{EBDSVTCDD2}
\end{equation}   
where 
\begin{displaymath} 
 a_{k,i}(x,t) = \displaystyle\sum_{j=1}^n \frac{\partial \psi_k(t)y}{\partial y_j} \
. \ \frac{\partial \psi_i(t)y}{\partial y_j}= \nabla_y \psi_k
. \nabla_y \psi_i  \ \ , \ y=\phi(t)x   
\end{displaymath}
and 
\begin{displaymath}
s_i(x,t)=\displaystyle\sum_{j=1}^n \frac{\partial^{2}
  \psi_i(t)y}{\partial y_j^{2}}= \Delta_y \psi_i(t)y \ \ \  \
y=\phi(t)x . 
\end{displaymath}

\end{proposition}
\begin{proof}
Clearly  $div_y(\vec{J})=-k\Delta T(y,t)$ for $y\in\Omega(t)$ and we
get (\ref{EBDSVTCDD1}). Now for (\ref{EBDSVTCDD2}), we have from
(\ref{amankaya}), 
\begin{displaymath}
\nabla_y T(y,t) =\nabla_x \overline{T}(x,t).D\psi(t)y  . 
\end{displaymath}
Hence, 
\begin{displaymath}
div_y(-k\nabla_y T(y,t))=-k \ \ div_y(\nabla_x \overline{T}(x,t) D\psi(t)y) =
\end{displaymath}
\begin{displaymath}
-k \ div_y \left( \displaystyle\displaystyle\sum_{i=1}^n \frac{\partial \ \overline{T}(x,t)}{\partial x_i}\ . \  \frac{\partial  \psi_i(t)y}{\partial y_1} ,...,\displaystyle\sum_{i=1}^n \frac{\partial \ \overline{T}(x,t)}{\partial x_i} \ . \  \frac{\partial \psi_i(t)y}{\partial y_n} \right). 
\end{displaymath}
Now observe that 
\begin{equation} \label{xime1}
\frac{\partial}{\partial y_j}\left(
  \displaystyle\displaystyle\sum_{i=1}^n \frac{\partial \
    \overline{T}(x,t)}{\partial x_i} . \frac{\partial \
    \psi_i(t)y}{\partial y_j}  \right)  
\displaystyle\sum_{i=1}^n  \left( \displaystyle\frac{\partial}{\partial y_j} \left( \displaystyle\frac{\partial \ \overline{T}(x,t)}{\partial x_i} \right) . \frac{\partial \ \psi_i(t)y}{\partial y_j} +  \frac{\partial \ \overline{T}(x,t)}{\partial x_i} .  \frac{\partial^{2} \ \psi_i(t)y}{\partial y_j^{2}} \right)  
\end{equation}
and by  (\ref{amankaya}), we get 
\begin{displaymath}
\frac{\partial}{\partial y_j} \left( \displaystyle\frac{\partial \
    \overline{T}(x,t)}{\partial x_i} \right) =
\displaystyle\sum_{k=1}^n \frac{\partial^{2} \
  \overline{T}(x,t)}{\partial x_k \ \partial x_i} \frac{\partial \
  \psi_k(t) y}{\partial y_j}  
\end{displaymath}
and we get in  (\ref{xime1})
\begin{displaymath}
\displaystyle\sum_{i=1}^n\left( \displaystyle\left( \displaystyle \ \displaystyle\sum_{k=1}^n \frac{\partial^{2} \ \overline{T}(x,t)}{\partial x_k \partial x_i}\ .  \  \frac{\partial \psi_k(t)y }{\partial y_j} \right) \frac{\partial \psi_i(t)y }{\partial y_j}  + \frac{\partial \ \overline{T}(x,t)}{\partial x_i} \ . \   \frac{\partial^2  \psi_i(t)y}{\partial y_j^2} \right) .
\end{displaymath}
Therefore 
\begin{displaymath}
-div_y(-k\nabla_y T(y,t)) = 
-k \displaystyle\sum_{j=1}^n\displaystyle \displaystyle\sum_{i=1}^n
\displaystyle\left( \displaystyle \ \displaystyle\sum_{k=1}^n
  \frac{\partial^{2} \ \overline{T}(x,t)}{\partial x_k \partial x_i} \
  .\  \frac{\partial \psi_k(t)y}{\partial y_j}\right)\ . \
\frac{\partial \psi_i(t)y}{\partial y_j} \  -      
\end{displaymath}
\begin{displaymath}
k \displaystyle\sum_{j=1}^n\displaystyle \displaystyle\sum_{i=1}^n
\frac{\partial \ \overline{T}(x,t)}{\partial x_i} \frac{\partial^{2}
  \psi_i(t)y} {\partial y_j^{2}}
\end{displaymath}
which leads to 
\begin{displaymath}
-k \ div_y \left( \displaystyle\displaystyle\sum_{i=1}^n \frac{\partial \ \overline{T}(x,t)}{\partial x_i}\ . \  \frac{\partial  \psi_i(t)y}{\partial y_1} ,...,\displaystyle\sum_{i=1}^n \frac{\partial \ \overline{T}(x,t)}{\partial x_i} \ . \  \frac{\partial \psi_i(t)y}{\partial y_n} \right) = 
\end{displaymath} 
\begin{displaymath}
-k \displaystyle\sum_{k,i=1}^n \frac{\partial^{2} \ \overline{T}(x,t)}{\partial x_k \partial x_i} \ .\ \left( \displaystyle\displaystyle\sum_{j=1}^n \frac{\partial \psi_k(t)y}{\partial y_j} \ . \ \frac{\partial \psi_i(t)y}{\partial y_j} \right) \ - \ k \displaystyle\sum_{i=1}^n \frac{\partial \overline{T}(x,t)}{\partial  x_i} \ . \ \left( \displaystyle\displaystyle\sum_{j=1}^n \frac{\partial^{2} \psi_i(t)y}{\partial y_j^{2}}\right).    
\end{displaymath}
and we get the result. 
\end{proof}

Concerning the main part in (\ref{EBDSVTCDD2}) we have the following 
\begin{proposition}

With the notations above, the term 
\begin{displaymath}
\displaystyle\sum_{k,i=1}^n  a_{k,i}(x,t) \frac{\partial^{2} \
  \overline{T}}{\partial x_k \partial x_i} (x,t) 
\end{displaymath}
can be written in divergence form. 
\end{proposition}
\begin{proof}
Just note that 
\begin{displaymath} \label{xime2}
\displaystyle\sum_{k,i=1}^n a_{k,i}(x,t) \frac{\partial^{2} \
  \overline{T}(x,t)}{\partial x_k \partial x_i} =
\displaystyle\sum_{i=1}^n \frac{\partial}{\partial x_i}
\left(\displaystyle\displaystyle\sum_{k=1}^n \frac{\partial
    \overline{T}(x,t)}{\partial x_i} . a_{k,i}(x,t) \right) -
\displaystyle\sum_{i=1}^n \frac{\partial \overline{T}(x,t)}{\partial
  x_i}\ . c_i(x,t) 
\end{displaymath} 
with 
\begin{displaymath}
c_i(x,t)= \displaystyle\sum_{k=1}^n \frac{a_{k,i}(x,t)}{\partial x_i}
\end{displaymath}
\end{proof}

\begin{remark}
Note that now  (\ref{EBDSVTCDD2}) can be written as 
\begin{displaymath} \label{70}
\left\{
\begin{array}{cl}
\frac{\partial \overline{T} (x,t)}{\partial t} +\overline{T}(x,t)  \overline{div (\vec{V})}(x,t) - k \left( div(B(x,t)) - \displaystyle\sum_{i=1}^n \frac{\partial \overline{T}(x,t)}{\partial x_i} d_i(x,t) \right) = \overline{f}(x,t) &\mbox{ } \\
&\mbox{ } \\
 \overline{T}(x,t)= 0 \  x\in \partial\Omega_0 \ \  \forall t  \ \ \ \ \ \ \ \  \overline{T}(x,0)= T_0(x)  \ \ \ \   x\in \Omega_0 &\mbox{ }  
\end{array}\right.
\end{displaymath}
with 
\begin{displaymath} \label{xime3}
\vec{B}=(B_i)_{i=1,..,n} = A(x,t) \nabla_x \overline{T}(x,t)  
\end{displaymath}
\begin{displaymath}
A(x,t)= (a_{k,i}(x,t)), \quad 
d_i(x,t)=s_i(x,t)-c_i(x,t) . 
\end{displaymath} 
\end{remark}

\section{Parabolic PDEs in moving domains}

Now we consider general parabolic equations in moving domains. That
means that the equations are not necessarily balance equations. Hence,
we consider 
\begin{equation}
\left\{
\begin{array}{cl}
\frac{\partial }{\partial t} T (y,t) - k \Delta_y T(y,t) + \displaystyle\sum_{i=1}^n
\frac{\partial T}{\partial y_i}(y,t) \ . \ g_i(y,t) +
c(y,t)T(y,t)=f(t,y) \ \ \ \ y\in \Omega(t) &\mbox{ } \\ 
\  \ \ \ T(y,t)= 0    \ \ y \in \partial \Omega(t)\ \ \  \forall t \ \  \ \ \ \ \ \ \ \  T(y,0)= T_0(y)  \ \ \ \ \  y \in \Omega_0     &\mbox{ }  
\end{array}\right.
\label{ESODV1}
\end{equation}
with $k>0$ and given smooth  $c(y,t)$ and $\vec{g}(y,t)= (g_i(y,t),
\ldots, g_n(y,t))$. Note that this equation contains
(\ref{EBDSVTCDD1}) as a particular case. 

Then we have the following result whose proof
follows from the computation in the sections above. 

\begin{proposition}
With the notations above (\ref{ESODV1}) is equivalent to 
\begin{equation}
\left\{
\begin{array}{c}
\displaystyle \frac{\partial \overline{T}}{\partial t}(x,t) - k div(B(x,t)) 
+ \nabla_x \overline{T}(x,t) . \big( \vec{h}(x,t)- \vec{d}(x,t)\big)  +
\overline{c}(x,t)  \overline{T}(x,t)=\overline{f}(x,t)  \ \ \ x\in
\Omega_0 \\ 
\  \ \ \ \overline{T}(x,t)= 0    \ \ x \in \partial \Omega_0\ \ \
\forall t \ \  \ \ \ \ \ \ \ \  \overline{T}(x,0)= T_0(x)  \ \ \ \ \
x \in \Omega_0       
\end{array}\right.
\label{ESODV2}
\end{equation}
with  
\begin{displaymath}
B(x,t)= A(x,t) \nabla_x \overline{T}(x,t), \quad A(x,t)=
(a_{k,i}(x,t)) , \quad 
 a_{k,i}(x,t) =  \displaystyle\sum_{j=1}^n \frac{\partial
   \psi_k(t)y}{\partial y_j} \ . \ \frac{\partial \psi_i(t)y}{\partial
   y_j} ,
\end{displaymath}
\begin{displaymath}
d_i(x,t)=s_i(x,t)-c_i(x,t), \quad s_i(x,t)= \Delta_y \psi_i(t)y, \quad
 c_i(x,t)= \displaystyle\sum_{k=1}^n \frac{a_{k,i}(x,t)}{\partial x_i}
\end{displaymath}
\begin{displaymath}
  \vec{h}(x,t)  = (\overline{\vec{g}}(x,t) \ .  \ \nabla_y \psi_1(t) y
  , \ldots, \overline{\vec{g}}(x,t) \ .  \ \nabla_y \psi_n(t) y ) ,
  \quad y=\phi(t)x . 
\end{displaymath}
\end{proposition}

Now we are in a position to proof that (\ref{ESODV1}) is well posed.

\begin{proposition}

Under the assumptions above, if the initial data satisfies 
\begin{displaymath}
T_0 \in L^{2}(\Omega_0)
\end{displaymath}
then (\ref{ESODV2}) and (\ref{ESODV1}) have a unique solution. 

\end{proposition}
\begin{proof}
Observe that in (\ref{ESODV2}) 
\begin{displaymath}
A(x,t)=D\psi(t)y .(D\psi(t)y)^{t} \ \ \ \ \ \  y = \phi(t)x . 
\end{displaymath}
Then we show below that this is a positive definite matrix. In fact
for  $\xi \in \R^{n}$,  $\xi \neq 0$, we have  
\begin{displaymath}
<A(x,t) \xi,\xi> = <(D\psi(t)y)^{t}\xi,(D\psi(t)y)^{t}\xi>=  \left \|
  (D\psi(t)y)^{t} \xi \right \|^{2} > 0 . 
\end{displaymath}
since   $(D\psi(t)y)^{t}$ is non singular. Also, from  (\ref{ACDC}),
the eigenvalues of  $D\phi(t)$ are bounded and bounded away from $0$
for all  $t \in [0,T]$ and so are the eigenvalues of
$D\psi(t)$. Therefore there exist  $\alpha = \alpha(T)> 0$ such that  
	$\left \| (D\psi(t)y)^{t} \xi \right \|^{2}\geq \alpha \left
          \| \xi \right \|^{2}$. 

Using this, the smoothness of the coefficients and the results in
\cite{HA, HA1}, we get that (\ref{ESODV2}) has a unique
smooth solution and so does (\ref{ESODV1}).
\end{proof}

\section{Maximum principle}

In this section we show that the parabolic equations in moving domains possess
the  maximum principle. We will show this on the particular example of
the heat equation 
\begin{equation}
\left\{
\begin{array}{cl}
\frac{\partial T }{\partial t}(y,t) - \Delta T(y,t) + a(y,t) \ T(y,t) = 0  \ \ \ \ \ y \in \Omega(t) &\mbox{ } \\
&\mbox{ } \\
\  \ \  \ \ T(y,t)= 0    \ \ y \in \partial \Omega(t)\ \ \  \forall t \ \  \ \ \ \ \ \ \ \ T(y,0)= T_0(y)  \ \ \ \ \  y \in \Omega_0     &\mbox{ } 
\end{array}\right.
\label{frank}
\end{equation}
witha a sufficiently smooth coefficient  $a(y,t)$. Then we have

\begin{proposition}

With the assumption above, if 
\begin{displaymath}
T_0\in L^{2}(\Omega_0), \quad T_0(x)\geq 0 \quad x\in \Omega_{0}
\end{displaymath}
and 
\begin{displaymath}
\alpha(t) \leq  a(y,t) \ \ \ \forall y\in \Omega(t) \ \ \ \forall t 
\end{displaymath}
for some smooth  $\alpha (t)$. Then  
\begin{displaymath}
T(y,t) \geq 0, \quad y \in \Omega(t), \quad t\geq 0 . 
\end{displaymath}

\end{proposition}
\begin{proof}
We multiply  (\ref{frank}) by the negative part of $T$,  $T^{-}(y,t)$, and integrate in
$\Omega(t)$,  to get 
\begin{displaymath}
\displaystyle\int_{\Omega(t)}^{} \frac{\partial T}{\partial t}(y,t) \
. \ T^{-}(y,t) \, dy - \displaystyle\int_{\Omega(t)}^{} \Delta T(y,t)\
. \ T^{-}(y,t) \, dy + \displaystyle\int_{\Omega(t)}^{} a(y,t) .\
T(y,t) . \ T^{-}(y,t) \, dy = 0 . 
\end{displaymath} 
Using (\ref{EBDSV5}) for $(T^{-})^{2}$ and the fact that 
$T^{-}(y,t)= 0$ in $\partial\Omega(t)$, because $T(y,t)=0$ in
$\partial\Omega(t)$, we have  
\begin{displaymath}
\frac{1}{2}\frac{d}{ d t} \displaystyle\int_{\Omega(t)}^{}
(T^{-})^{2}(y,t) \, dy + \displaystyle\int_{\Omega(t)}^{}  \left |
  \nabla T^{-}(y,t) \right |^{2} \, dy +
\displaystyle\int_{\Omega(t)}^{} a(y,t) (T^{-})^{2}(y,t) \, dy= 0 . 
\end{displaymath} 
Hence 
\begin{displaymath} \label{rena}
\frac{1}{2}\frac{d}{ d t} \left \| T^{-}(.,t) \right \|^{2}_{L^{2}(\Omega(t))}+ \alpha(t) \left \| T^{-}(.,t) \right \|^{2}_{L^{2}(\Omega(t))} \leq 0
\end{displaymath}
and taking $\bar{F}(t)=\left \| T^{-}(.,t) \right
\|^{2}_{L^{2}(\Omega(t))}$, we have 
\begin{displaymath}
\frac{d }{d t} \bar{F} (t) + 2 \alpha(t)  \bar{F}(t) \leq 0
\end{displaymath} 
and Gronwall's lemma leads to $\bar{F}(t) \leq \left \| T^{-}_0 \right \|_{L^2(\Omega_0)} e^{-2
  \int_{0}^{t} \alpha(s)\, ds} =0$, 
since $T_0^{-}= 0$ in $\Omega_0$. Therefore $T^{-}(y,t) = 0$ for $y\in
\Omega(t)$ and $t\geq 0$ as claimed. 
\end{proof}

\section{Energy estimates}

In this section we derive suitable energy estimates for the heat
equation in a moving domain 
\begin{equation}
\left\{
\begin{array}{cl}
\frac{\partial T }{\partial t}(y,t) - \Delta T(y,t) + a(y,t) \ T(y,t) = 0  \  \ \ y\in\Omega(t) &\mbox{ } \\
&\mbox{ } \\
\  \ \  \ \ T(y,t)= 0    \ \ y \in \partial \Omega(t)\ \ \  \forall t \ \  \ \ \ \ \ \ \ \ T(y,0)= T_0(y)  \ \ \ \ \  y \in \Omega_0     &\mbox{ } 
\end{array}\right.
\label{leo2}
\end{equation}
with a smooth enough  $a(y,t)$. First, we have for nonnegative
solutions

\begin{proposition}
Assume  
\begin{displaymath}
T_0 \in L^{2}(\Omega_0)  \quad  T_0(x)\geq 0, \quad x\in \Omega_{0}
\end{displaymath}
and 
\begin{displaymath}
\alpha(t) \leq  a(y,t) \ \ \ \forall y\in \Omega(t) \ \ \ \forall t .
\end{displaymath}
for some smooth  $\alpha (t)$ such that 
\begin{displaymath}
\liminf_{t\rightarrow \infty} \frac{1} {t}
\displaystyle\int_{0}^{t} \alpha(s)\, ds > \alpha_0 > 0 . 
\end{displaymath}
Then 
\begin{displaymath}
\displaystyle\int_{\Omega(t)}^{} T (y,t) \, dy  \leq  e^{-\int_{0}^{t}
  \alpha (t) \, ds} \displaystyle\int_{\Omega_0}^{} T_0(x) \, dx
\xrightarrow[{t \to \infty }]{} 0 . 
\end{displaymath}

\end{proposition}
\begin{proof}
From  (\ref{EBDSV5}) 
\begin{displaymath}
\frac{d}{d t} \displaystyle\int_{\Omega(t)}^{} T(y,t)\,
dy = \displaystyle\int_{\Omega(t)}^{} \frac{\partial T}{\partial
  t}(y,t) \, dy +\displaystyle\int_{\partial \Omega(t)}^{} T(y,t)
\vec{V}(y)   \, d\vec{s}
\end{displaymath}
and since $T$ vanishes on the boundary, we have  
\begin{displaymath} \label{fer1}
\frac{d}{d t} \displaystyle\int_{\Omega(t)}^{} T(y,t)\,
dy  =\displaystyle\int_{\Omega(t)}^{} \frac{\partial T}{\partial
  t}(y,t) \, dy. 
\end{displaymath}
Using this, we integrate in  (\ref{leo2}) in  $ \Omega(t)$, to get 
\begin{displaymath}
\displaystyle\int_{\Omega(t) }^{} \frac{\partial T}{\partial t}(y,t)
\, dy - \displaystyle\int_{\Omega(t)}^{} \Delta T(y,t) \, dy +
\displaystyle\int_{\Omega(t)}^{} a(y,t) \ T(y,t) \, dy=0 .  
\end{displaymath}
Now Green's formula leads to 
\begin{displaymath}
\frac{d }{d t}\displaystyle\int_{\Omega(t)}^{} T (y,t) \, dy -
\displaystyle\int_{\partial\Omega(t)}^{} \frac{\partial T}{\partial
  \vec{n}} (y,t)\, d\vec{s} + \displaystyle\int_{\Omega(t)}^{} a(y,t)
\ T(y,t) \, dy = 0.  
\end{displaymath}

By the maximum prinicple we know that $T(y,t) \geq 0$ for $y\in
\Omega(t)$ and $t\geq 0$, and then for $y \in \partial\Omega(t)$ we
have $\frac{\partial T}{\partial \vec{n}} (y,t) \leq 0 $ and then 
\begin{displaymath}
\frac{d }{d t}\displaystyle\int_{\Omega(t)}^{} T (y,t) \, dy +
\displaystyle\int_{\Omega(t)}^{} a(y,t)  \ T(y,t) \, dy \leq 0. 
\end{displaymath}

Hence, denoting $\bar{Y}(t)=\displaystyle\int_{\Omega(t)}^{} T (y,t)
\, dy$ we have  
\begin{displaymath}
\frac{d \bar{Y} }{d t}(t) + \alpha(t) \bar{Y}(t) \leq 0 
\end{displaymath}
and Gronwall's lemma gives 
\begin{displaymath}
\bar{Y}(t)= \displaystyle\int_{\Omega(t)}^{} T (y,t) \, dy  \leq
e^{-\int_{0}^{t} \alpha (t) \, ds} \displaystyle\int_{\Omega_0}^{}
T_0(x) \, dx \xrightarrow[{t \to \infty }]{} 0.  
\end{displaymath}
since, by assumption  
\begin{displaymath}
\liminf_{t\rightarrow \infty} \frac{1} {t}     \displaystyle\int_{0}^{t} \alpha(s)\, ds > \alpha_0 > 0
\end{displaymath}
and then 
\begin{displaymath}
e^{-\int_{0}^{t} \alpha(s)\, ds}=e^{-t\left(\displaystyle\frac{1}{t}
    \int_{0}^{t} \alpha(s)\, ds\right)} \leq  e^{-\alpha_0
  t}\xrightarrow[{t \to \infty }]{} 0. 
\end{displaymath}
for $ t>> 1$. 
\end{proof}

Now without assuming sign on the solutions, we have 

\begin{proposition}

With the notations above, assume 
\begin{displaymath}
T_0 \in L^{2}(\Omega_0) 
\end{displaymath}
and the function 
\begin{displaymath}
\gamma(t) =  \alpha(t)  - C_0(\Omega(t)) , 
\end{displaymath}
is such that for some $\alpha_{1}>0$, 
\begin{displaymath}
\liminf_{t\rightarrow \infty} \frac{1} {t}
\displaystyle\int_{0}^{t} \gamma(s)\, ds > \alpha_1 > 0 , 
\end{displaymath}
where  $C_0(\Omega(t))$ is the  Poncair\`e constant in  $\Omega(t)$. 

Then 
\begin{displaymath}
0 \leq \displaystyle\int_{\Omega(t)}^{} T^2(y,t)\, dy \leq e^{- 2
  \int_{0}^{t} \gamma(s) \, ds} \displaystyle\int_{\Omega(t)}^{}
T^{2}_0(x)  \, dx  \xrightarrow[{t \to \infty }]{} 0 . 
\end{displaymath} 

\end{proposition}
\begin{proof}
Multiply  (\ref{leo2}) by  $T(y,t)$ and integrate in $\Omega(t)$, to get  
\begin{displaymath}
\displaystyle\int_{\Omega(t)}^{} \frac{\partial T}{\partial t}(y,t)
T(y,t) \, dy - \displaystyle\int_{\Omega(t)}^{} \Delta T(y,t) T(y,t)
\, dy + \displaystyle\int_{\Omega(t)}^{} a(y,t) \ T^{2}(y,t) \, dy = 0 . 
\end{displaymath}
Using (\ref{EBDSV5}), the boundary conditions and the Green's formula
we have 
\begin{displaymath}
\frac{1}{2}\frac{d}{ d t} \displaystyle\int_{\Omega(t)}^{} T^{2}(y,t) \, dy + \displaystyle\int_{\Omega(t)}^{}  \left | \nabla T(y,t) \right |^{2} \, dy + \displaystyle\int_{\Omega(t)}^{} a(y,t) T^2(y,t) \, dy= 0 .
\end{displaymath}
Now the Poincar\`e inequality in  $\Omega(t)$ gives for 
any  smooth function vanishing on $\partial \Omega(t)$, 
\begin{displaymath}
\left \| \nabla u \right \|^{2}_{L^{2}(\Omega(t))} \geq
C_0(\Omega(t)) \left \| u \right \|^{2}_{L^{2}(\Omega(t))}  . 
\end{displaymath}
This and the assumption on $a(y,t)$ leads to 
\begin{equation} 
\frac{1}{2}\frac{d}{ d t} \left \| T(.,t) \right
\|^{2}_{L^{2}(\Omega(t))}+ \gamma(t) \left \| T(.,t) \right
\|^{2}_{L^{2}(\Omega(t))} \leq 0. 
\label{fercho}
\end{equation}
Thus, denoting $\bar{Z}(t)=\left \| T(.,t) \right
\|^{2}_{L^{2}(\Omega(t))}$,  (\ref{fercho}) reads 
\begin{displaymath}
\frac{d }{d t} \bar{Z} (t) + 2 \gamma(t)  \bar{Z}(t) \leq 0
\end{displaymath} 
and Gronwall's lemma yields 
\begin{displaymath}
\bar{Z}(t) \leq \left \| T_0 \right \|_{L^2(\Omega_0)}^{2} e^{-2
  \int_{0}^{t} \gamma(s)\, ds} \xrightarrow[{t \to \infty }]{} 0. 
\end{displaymath}
since by assumption 
\begin{displaymath}
\liminf_{t\rightarrow \infty} \frac{1} {t}
\displaystyle\int_{0}^{t} \gamma(s)\, ds > \alpha_1 > 0 
\end{displaymath}
and then 
\begin{displaymath}
e^{-\int_{0}^{t} \gamma(s)\, ds}=e^{-t\left(\displaystyle\frac{1}{t}
    \int_{0}^{t} \gamma(s)\, ds\right)} \leq  e^{-\alpha_1
  t}\xrightarrow[{t \to \infty }]{} 0
\end{displaymath}
for $ t>> 1$. 
\end{proof}


\begin{thebibliography}{9}{}
\bibitem{HA}  Amann, H.  \emph{Linear and Quasilinear Parabolic Problems}, \textbf{89}, Birkhauser Verlag, Berlin (1995).

\bibitem{HA1}  Amann, H. \emph{Nonhomogeneous linear and quasilinear elliptic and parabolic boundary value problems
Function spaces, differential operators and nonlinear analysis},
(Friedrichroda, 1992), Teubner, 1993, 133, 9-126 

\bibitem{DG} 	Duvaut, G. \emph{Mécanique des milieux continus},  Masson, 1990

\bibitem{H} Hartman, P. \emph{Ordinary differential equations}, John
  Wiley and sons, 1964. 


\bibitem{CAJ} Chorin, A.  J., Marsden J.E.,  \emph{A mathematical
    introduction to fluid mechanics},  Springer,  1979.

\bibitem{LOA} Ladyzhenskaya, O. A., \emph{The mathematical theory of
    viscous incompressible flow},  Gordon and Breach, 1969.



\end{thebibliography}
\end{document}